\documentclass[11pt,a4paper,reqno,draft]{amsart}

\usepackage[left=3cm,right=3cm,top=3cm,bottom=2cm,foot=1.7cm]{geometry}
\usepackage{amsbsy,amsmath,amssymb,amsthm,mathrsfs}
\usepackage{amsaddr}

\usepackage{enumerate}
\usepackage{color}
\newtheorem{theorem}{Theorem}
\newtheorem{theoremx}{Theorem}

\newtheorem{lemma}{Lemma}

\theoremstyle{definition}
\newtheorem{definition}{Definition}
\newtheorem*{remark}{Remark}

\title[Barab\'asi--Albert random graph with multiple type edges and perturbation]{Barab\'asi--Albert random graph with multiple type edges with perturbation}

\author{\'Agnes Backhausz}
\address{
	ELTE E\"otv\"os Lor\'and University, Budapest, Hungary\\
	Faculty of Science\\
	Department of Probability Theory and Statistics\\
	and\\
	Alfr\'ed R\'enyi Institute of Mathematics, Budapest, Hungary}
\email{agnes.backhausz@ttk.elte.hu}

\author{Bence Rozner}
\address{
	ELTE E\"otv\"os Lor\'and University, Budapest, Hungary\\
	Faculty of Science\\
	Department of Probability Theory and Statistics}
\email{bence.rozner@ttk.elte.hu}


\date{\today}

\begin{document}
\maketitle
\thispagestyle{empty}

\begin{abstract}
	In this paper we introduce the perturbed version of the Barab\'asi--Albert random graph with multiple type edges and prove the existence of the (generalized) asymptotic degree distribution. Similarly to the non-perturbed case, the asymptotic degree distribution depends on the almost sure limit of the proportion of edges of different types. However, if there is perturbation, then the resulting degree distribution will be deterministic, which is a major difference compared to the non-perturbed case.
\end{abstract}

\section{Introduction}
Many preferential attachment random graph models have been studied for a long time, see e.g.\ \cite{Barabasi_Albert,Durett,Frieze_Karonski,van_der_Hofstad}. This is mainly motivated by the analysis of real-world networks, such as the internet and different kind of biological or social networks. In many applications, the preferential attachment graphs can be extended by various features. For example we can assign types to the vertices or to the edges of the graph, which results in more adequate models. In particular, in a social network we can distinguish men and women, or, like in the models of population dynamics, individuals can be divided into different groups according to a genetic, physical property, or certain behaviour. There are various random graph models with multi-type vertices, see e.g.\ \cite{Abdullah,Antunovic,Rosengren}, which have been investigated. In all of these models, the types of vertices are chosen with a dynamics strongly related to the evolution of the graph. That is, there is an interaction between the choice of new edges and the type of new vertices, and the structure of the graph has an impact on the proportion of vertices of different types. For example, to answer the question of coexistence (which is a common question in population dynamics models), that is, to decide whether the proportion of all types of vertices tends to a positive number or not, one has to understand quantities like the number of vertices of a given type with a given number of edges. This kind of analysis is performed in the papers mentioned above.

Now we consider models where it is not the vertices, but the edges that have different types. This can be used to model different kind of relationships between the individuals -- in a social, biological or financial network, connections are usually not of the same nature, and this can be important from the point of view of contagion or epidemic spread. In our model, the type of an edge is an element of a fixed finite set, and it does not change with time. It is chosen randomly when the edge is born, with a distribution that depends on the current state of the graph and on the types of the edges going out from the neighbours of the new vertex. A general family of preferential attachment random graphs with multi-type edges has been examined in our previous work \cite{Backhausz_Rozner}. Growing networks with two different types of edges can be considered as directed graphs. In this case the type of an edge is its orientation. That is, when a new vertex is born, then it is attached to the graph with an edge directed from the new vertex to the already existing ones or directed from the existing vertices to the new one, and this corresponds to the two different types. Directed preferential attachment models were introduced and examined in \cite{Bollobas_Borgs,Wang}, but with different dynamics than in \cite{Backhausz_Rozner}.

In this paper we are interested in a version of robustness in preferential attachment graph models with multi-type edges. The goal is to compare (i) a model in which the probability of choosing a type is exactly the proportion of the current type among the edges going out from the endpoint of the new edge; and (ii) its modified version, when, after this step, types can change with certain probability. In particular, we introduce perturbation in the multi-type Barab\'asi--Albert random graph, and prove that this shows different phenomena than the original version. That is, errors in the dynamics of multi-type random graphs can lead to essential changes in the asymptotic behaviour of the model. The multi-type Barab\'asi--Albert random graph model has been described in \cite{Backhausz_Rozner}, which is a generalization of the Barab\'asi--Albert random graph model introduced in \cite{Barabasi_Albert}, specified in \cite{Bollobas}.

We prove the existence of the asymptotic degree distribution in the perturbed Barab\'asi--Albert random graph, and we also provide recurrence equations for the asymptotic degree distribution. The main difference between the perturbed and the non-perturbed Barab\'asi--Albert random graph is the deterministic or stochastic nature of the asymptotic degree distribution. The reason for that is the asymptotic behaviour of the proportion of edges of different types which can be described with an urn model. If there is no perturbation, then the proportion of edges of a given type converges to a non-degenerate random variable. On the other hand, if there is perturbation, then it converges to a deterministic constant almost surely. This is based on the properties of the underlying urn models which is explained in more details in \cite{Laruelle_Pages}. In the current paper, we generalize the results of \cite{Laruelle_Pages} about the almost sure limit of the proportion of edges of different types (or colours) for the case when we also allow multiple drawings with replacement.

Ostroumova, Ryabchenko and Samosvat \cite{Ostroumova} propose a general class of preferential attachment models with single-type edges. They also introduce perturbation in the dynamics, which is different from the one that we have in our model. They assume that the error terms converge to zero with rate $O(1/n)$, where $n$ is the size of the graph. In the perturbed (multi-type) Barab\'asi--Albert random graph, we assume that the probability of errors converges to a positive number.

\textbf{Outline.\ }In Section 2, we introduce the perturbed Barab\'asi--Albert random graph and then formulate the main result of this paper. In Section 3, we generalize a result on an urn model by Laruelle and Pag\`es \cite{Laruelle_Pages}. By using the asymptotic properties of this general urn model, we are able to prove the main theorem. The proof is elaborated in Section 4, by showing that the conditions of a general theorem from \cite{Backhausz_Rozner} hold.

\section{The model and main results}
\subsection{Notation} Let $(G_{n})_{n=0}^{\infty}$ be a sequence of finite random multi-graphs without loop edges. For all $n\geq{}0$, the set of vertices and the set of edges of $G_{n}$ are denoted by $V_{n}$ and $E_{n}$, respectively. Throughout the paper, $N$ will denote the number of possible types of edges. For every $l\in[N]=\{1,2,\dots,N\}$ let $E_{n}^{(l)}$ denote the set of edges of type $l$ in $G_{n}$. We assume that the different types form a partition of the edges. For every $l$ we have $E_{n}^{(l)}\subseteq{}E_{n+1}^{(l)}$. We assume that the initial configuration $G_{0}$ is a finite deterministic graph, moreover for every $l\in[N]$ we have $|E_{0}^{(l)}|>0$. We group the vertices based on the number of edges of different types connected to them.
\begin{definition}
	For a given $n$, the generalized degree of a vertex $v\in{}V_{n}$ in the $n^{\textrm{th}}$ step is $\textbf{deg}_{n}(v)=\left(\textrm{deg}_{n}^{(l)}(v),l\in[N]\right)$, where $\textrm{deg}_{n}^{(l)}(v)$ is the number of edges of type $l$ connected to $v$ in $G_{n}$. For every $\boldsymbol{d}=(d_{1},\dots,d_{N})\in\mathbb{N}^{N}$ we define $X_{n}(\boldsymbol{d})=|\{v\in{}V_{n}:\textbf{deg}_{n}(v)=\boldsymbol{d}\}|$, i.e.\ the number of vertices in $G_{n}$ with generalized degree $\boldsymbol{d}$, that is, the number of vertices with $d_{l}$ edges of type $l$.
\end{definition}
 Finally, for every $n$, the $\sigma$-algebra generated by the first $n$ multi-type graphs is denoted by $\mathcal{F}_{n}$. We can choose $\mathcal{F}_{0}$ to be the trivial $\sigma$-algebra, since $G_{0}$ is deterministic.

Throughout the paper, for every vector $\boldsymbol{v}$ of finite dimension we define the weight of $\boldsymbol{v}$ as the sum of all the elements of $\boldsymbol{v}$. This will be denoted by $s(\boldsymbol{v})$.

In the sequel, $\mathbb{N}$ will denote the set of non-negative integers, furthermore $\boldsymbol{e}_{l}$ will denote the $l^{\textrm{th}}$ unit vector in $\mathbb{R}^{N}$ and $\boldsymbol{1}$ will be the vector with entries all equal to $1$.

\subsection{Assumptions.}In order to introduce perturbation, we will use the matrices of error probabilities denoted by $\boldsymbol{F}_{n}$. For every $n\geq{}1$ let $\boldsymbol{F}_{n}=(\varepsilon_{k,l}^{(n)}:k,l\in[N])\in[0,1]^{N\times{}N}$ be a matrix, such that for every fixed $k\in[N]$ we have $\sum_{l=1}^{N}\varepsilon_{k,l}^{(n)}=1$. That is, $\varepsilon_{k,l}^{(n)}$ is the probability that a type $k$ edge becomes type $l$ in the $n^{\textrm{th}}$ step. We assume that there is a matrix denoted by $\boldsymbol{F}=\left(\varepsilon_{k,l}:k,l\in[N]\right)\in[0,1]^{N\times{}N}$  such that for every fixed $k\in[N]$ we have $\sum_{l=1}^{N}\varepsilon_{k,l}=1$ and for every $k,l\in[N]$ we have $\varepsilon_{k,l}^{(n)}\to\varepsilon_{k,l}$ as $n\to\infty$. That is, for every $k,l$, the probability that a type $k$ edge becomes type $l$, converges to $\varepsilon_{k,l}$.

The dynamics of the perturbed Barab\'asi--Albert random graph is the following. Let us fix a positive integer denoted by $M$. In the $n^{\textrm{th}}$ step
\begin{enumerate}
	\item{}a new vertex $v_{n}$ is born.
	\item{}The vertex $v_{n}$ attaches to some of the already existing vertices with $M$ (not necessarily different) edges with probabilities proportional to the actual degrees of the existing vertices. The endpoints of the $M$ new edges are chosen independently. We do not update the degrees of the vertices until the end of the $n^{\textrm{th}}$ step.
	\item{}Every new edge gets a type randomly. The types of the new edges are chosen independently, and the probability of each type is its proportion among the types of the edges of the already existing endpoint of the new edge (not counting the edges added in the actual step).
	\item{}The types of the new edges change independently of each other with probabilities given by $\boldsymbol{F}_{n}$, i.e.\ if there is a new edge of type $k$, then its type after perturbation is $l$ with probability $\varepsilon_{k,l}^{(n)}$.
\end{enumerate}

\subsection{The main result.} We are now ready to state our main theorem on the asymptotic degree distribution of the perturbed Barab\'asi--Albert random graph.
\begin{theorem}
	\label{thm_asymptotic_degree_dist}
	In the perturbed Barab\'asi--Albert random graph, if we assume that $\boldsymbol{F}$ is irreducible, then for every $\boldsymbol{d}=(d_{1},\dots,d_{N})\in\mathbb{N}^{N}$
	\begin{align*}
		\lim_{n\to\infty}\frac{X_{n}(\boldsymbol{d})}{|V_{n}|}&=x(\boldsymbol{d})\textrm{ a.s.}
	\end{align*}
	holds for a deterministic $x(\boldsymbol{d})\in[0,1]$. Furthermore, for every $\boldsymbol{d}\in\mathbb{N}^{N}$ we have the following recurrence equation:
	\begin{align*}
		\textrm{if }s(\boldsymbol{d})=M\textrm{, then}\quad{}x(\boldsymbol{d})&=\frac{2\cdot{}M!}{M+2}\left[\prod_{l=1}^{N}\frac{1}{d_{l}!}\left(\sum_{k=1}^{N}\psi^{(k)}\cdot\varepsilon_{k,l}\right)^{d_{l}}\right]\\
		\textrm{if }s(\boldsymbol{d})>M\textrm{, then}\quad{}x(\boldsymbol{d})&=\sum_{l=1}^{N}\frac{(\boldsymbol{d}-\boldsymbol{e}_{l})^{T}\boldsymbol{F}_{\bullet,l}}{s(\boldsymbol{d})+2}x(\boldsymbol{d}-\boldsymbol{e}_{l}),
	\end{align*}
	where
	\begin{itemize}
		\item{}$s(\boldsymbol{d})=\boldsymbol{d}^{T}\boldsymbol{1}=\sum_{l=1}^{N}d_{l}$,
		\item{}$\psi^{(k)}$ is the almost sure limit of the proportion of edges of type $k$, which is a deterministic constant and
		\item{}$\boldsymbol{F}_{\bullet,l}$ denotes the $l^{\textrm{th}}$ column of the matrix $\boldsymbol{F}$.
	\end{itemize}
	
	If we also assume that $\boldsymbol{F}$ is symmetric, then for every $\boldsymbol{d}=(d_{1},\dots,d_{N})\in\mathbb{N}^{N}$ we have the following recurrence equation:
	\begin{align*}
		\textrm{if }s(\boldsymbol{d})=M\textrm{,  then}\quad{}x(\boldsymbol{d})&=\frac{2\cdot{}M!}{M+2}\cdot\frac{1}{\prod_{l=1}^{N}d_{l}!}\left(\frac{1}{N}\right)^{M}\\
		\textrm{if }s(\boldsymbol{d})>M\textrm{, then}\quad{}x(\boldsymbol{d})&=\sum_{l=1}^{N}\frac{(\boldsymbol{d}-\boldsymbol{e}_{l})^{T}\boldsymbol{F}_{\bullet,l}}{s(\boldsymbol{d})+2}x(\boldsymbol{d}-\boldsymbol{e}_{l}).
	\end{align*}
\end{theorem}
\begin{remark}
	Notice that $x(\boldsymbol{d})=0$ if $s(\boldsymbol{d})<M$ or if we have $d_{l}<0$ for any $l\in[N]$.
\end{remark}

For comparison, let us assume that there is no perturbation, i.e.\ $\boldsymbol{F}_{n}$ is the identity matrix for every $n$. It also means that the condition on the irreducibility of $\boldsymbol{F}$ fails in Theorem \ref{thm_asymptotic_degree_dist}. However, Theorem 2 in \cite{Backhausz_Rozner} describes the asymptotic degree distribution in the non-perturbed version of the model. This is the following: in the multi-type Barab\'asi--Albert random graph for every $\boldsymbol{d}\in\mathbb{N}^{N}$ we have
\begin{align*}
	\lim_{n\to\infty}\frac{X_{n}(\boldsymbol{d})}{|V_{n}|}&=x(\boldsymbol{d})\textrm{ a.s., where now $x(\boldsymbol{d})$ is a non-deterministic random variable.}
\end{align*}
The random variables $x(\boldsymbol{d})$ satisfy the following recurrence equation for every $\boldsymbol{d}\in\mathbb{N}^{N}$:
\begin{align*}
	x(\boldsymbol{d})&=\sum_{l=1}^{N}\frac{d_{l}-1}{s(\boldsymbol{d})+2}x(\boldsymbol{d}-\boldsymbol{e}_{l})+\frac{2}{s(\boldsymbol{d})+2}\mathbb{P}(M=s(\boldsymbol{d}))\frac{s(\boldsymbol{d})!}{\prod_{l=1}^{N}d_{l}!}\prod_{l=1}^{N}\left(\psi^{(l)}\right)^{d_{l}},
\end{align*}
where $\psi^{(l)}$ is the almost sure limit of the proportion of the edges of type $l$. In this case the asymptotic degree distribution is random, which means that it also depends on the asymptotic proportion of edges of different types. If $M=1$, that is, the graph is a tree, then $(\psi^{(l)},l\in[N])$ has Dirichlet distribution with parameters $(E_{0}^{(l)},l\in[N])$. However, in the perturbed Barab\'asi--Albert random graph the asymptotic degree distribution is deterministic.

\section{A general urn model}
In order to prove our main theorem, we need to understand the asymptotic behaviour of the composition of number of edges of different types. We introduce a general urn model to describe the proportion of the edges of the multi-type perturbed Barab\'asi--Albert random graph. This model is a generalization of a special case of the urn model introduced by Laruelle and Pag\`es in \cite{Laruelle_Pages}. We remark that the generalization of the results of \cite{Gangopadhyay_Maulik} could also be used for our purposes.

We assume that there are $N$ colours denoted by $\left\{1,2,\dots,N\right\}$. The composition vector of the urn in the $n^{\textrm{th}}$ step is denoted by $\boldsymbol{C}_{n}\in\mathbb{N}^{N}$, i.e.\ $C_{n,i}$ is the number of balls of colour $i$. The total number of balls in the urn in the $n^{\textrm{th}}$ step is denoted by $s(\boldsymbol{C}_{n})=\sum_{i=1}^{N}C_{n,i}$. We assume that in every step we draw $M$ balls, with replacement, independently of each other and at the end of the step we add some additional balls to the urn. In the $n^{\textrm{th}}$ step for trial $i$ (where $i\in[M]$), let $\boldsymbol{\chi}_{n}^{(i)}$ be the $N$ dimensional indicator vector of the colour drawn, and let $\boldsymbol{R}_{n}^{(i)}$ be the $N\times{}N$ dimensional replacement matrix with possibly random but non-negative entries. This means that $\left(\boldsymbol{R}_{n}^{(i)}\right)_{k,l}$ is the number of balls of colour $l$ added to the urn if a ball of colour $k$ was chosen in the $n^{\textrm{th}}$ step for trial $i$.

For every $n$, we have
\begin{align}
\label{eq_urn_dynamics}
	\boldsymbol{C}_{n+1}&=\boldsymbol{C}_{n}+\sum_{i=1}^{M}\boldsymbol{R}_{n+1}^{(i)}\boldsymbol{\chi}_{n+1}^{(i)}.
\end{align}

For every $n$, we denote by $\mathcal{G}_{n}$ the $\sigma$-algebra of the draws and replacements in the first $n$ steps, that is, the $\sigma$-algebra generated by
\begin{align*}
	\boldsymbol{C}_{0},\left(\boldsymbol{\chi}_{j}^{(i)},i=1,2,\dots,M\right)_{j=1}^{n}\textrm{ and }\left(\boldsymbol{R}_{j}^{(i)},i=1,2,\dots,M\right)_{j=1}^{n}.
\end{align*}
 
We have the following assumptions on the urn model:
\begin{description}
	\item[(U1)]the initial configuration $\boldsymbol{C}_{0}$ is non-negative and at least one of the coordinates is positive;
	\item[(U2)]for every $n$, $i$ and $j$ we have $\mathbb{P}\left(\boldsymbol{\chi}_{n}^{(i)}=\boldsymbol{e}_{j}\Big|\mathcal{G}_{n-1}\right)=\frac{C_{n-1,j}}{s(\boldsymbol{C}_{n-1})}$, where $\boldsymbol{e}_{j}$ is the $j^{\textrm{th}}$ unit vector in $\mathbb{R}^{N}$; that is, the probability of choosing a colour is its proportion in the urn;
	\item[(U3)]for every $n$, the random variables $\boldsymbol{\chi}_{n}^{(1)},\boldsymbol{\chi}_{n}^{(2)},\dots,\boldsymbol{\chi}_{n}^{(M)}$ are identically distributed given $\mathcal{G}_{n-1}$, and similarly the random matrices $\boldsymbol{R}_{n}^{(1)},\boldsymbol{R}_{n}^{(2)},\dots,\boldsymbol{R}_{n}^{(M)}$ are identically distributed given $\mathcal{G}_{n-1}$, furthermore we assume that
	\begin{align*}
		\boldsymbol{\chi}_{n}^{(1)},\boldsymbol{\chi}_{n}^{(2)},\dots,\boldsymbol{\chi}_{n}^{(M)},\boldsymbol{R}_{n}^{(1)},\boldsymbol{R}_{n}^{(2)},\dots,\boldsymbol{R}_{n}^{(M)}
	\end{align*}
	are conditionally independent given $\mathcal{G}_{n-1}$. This implies that, even if the replacement matrix is random, it is independent of the actual draw.
\end{description}

For every $n$, we define the generating matrices as the conditional expectation of the replacement matrix, i.e.\ $\boldsymbol{H}_{n}=\mathbb{E}\left(\boldsymbol{R}_{n}^{(1)}\Big|\mathcal{G}_{n-1}\right)$. We assume that
\begin{description}
	\item[(U4)]for every $n$ and $i$ the replacement matrix $\boldsymbol{R}_{n}^{(i)}$ has non-negative values almost surely;
	\item[(U5)]for every $n$ and $i$ every column of the replacement matrix has the same weight almost surely, i.e.\ for every $n$, $i$ and $j$ we have $s\left((\boldsymbol{R}_{n}^{(i)})_{\bullet,j}\right)=\gamma_{1}$, that is, the number of balls added to the urn is constant;
	\item[(U6)]every column of the generating matrices has the same weight almost surely, i.e.\ for every $n$ and $j$ we have $s\left((\boldsymbol{H}_{n})_{\bullet,j}\right)=\gamma_{2}$. This constant is also known as the balance of the urn.
\end{description}

Finally, we assume that
\begin{description}
	\item[(U7)]there exists an irreducible $N\times{}N$ matrix denoted by $\boldsymbol{H}$ such that
	\begin{align*}
		\boldsymbol{H}_{n}\xrightarrow[n\to\infty]{a.s.}\boldsymbol{H}.
	\end{align*}
\end{description}

We denote by $\boldsymbol{v}_{\boldsymbol{H}}^{*}$ the normalized eigenvector of $\boldsymbol{H}$ corresponding to the largest eigenvalue of $\boldsymbol{H}$ such that $||\boldsymbol{v}_{\boldsymbol{H}}^{*}||_{2}=1$.

\begin{remark}
	If $M=1$, then we get back the urn model in \cite{Laruelle_Pages}.
\end{remark}

The next theorem states that the asymptotic composition of colours can be described with the normalized eigenvector corresponding to the largest eigenvalue.
\begin{theorem}
	\label{thm_asymptotic_urn_composition}
	For all integers $M>0$, assumptions $\mathbf{(U1)}$-$\mathbf{(U7)}$ imply that
	\begin{align*}
		\frac{\boldsymbol{C}_{n}}{s(\boldsymbol{C}_{n})}\xrightarrow[n\to\infty]{a.s.}\boldsymbol{v}^{*}_{\boldsymbol{H}}.
	\end{align*}
\end{theorem}

To prove this theorem, we will use the same method as Laruelle and Pag\`es in \cite{Laruelle_Pages}, also known as the ordinary differential equation (ODE) method, which is a powerful tool of stochastic approximation.

Let us have a filtered probability space denoted by $(\Omega, (\mathcal{G}_{n})_{n\geq{}0},\mathbb{P})$ and consider the following recurrence equation:
\begin{align*}
\boldsymbol{\vartheta}_{n+1}&=\boldsymbol{\vartheta}_{n}-\gamma_{n+1}h(\boldsymbol{\vartheta}_{n})+\gamma_{n+1}\left(\Delta\boldsymbol{M}_{n+1}+\boldsymbol{r}_{n+1}\right)
\end{align*}
for $n\geq{}n_{0}$, where $h:\mathbb{R}^{N}\to\mathbb{R}^{N}$ is a locally Lipschitz continuous function, $\boldsymbol{\vartheta}_{n_{0}}$ is an $\mathbb{R}^{N}$-valued $\mathcal{G}_{n_{0}}$ measurable random variable, $(\gamma_{n})_{n\geq{}n_{0}+1}$ is a (deterministic) sequence of positive numbers, $\left(\Delta\boldsymbol{M}_{n}\right)_{n\geq{}n_{0}+1}$ is a martingale difference in $(\mathcal{G}_{n})_{n\geq{}n_{0}}$ and finally $(\boldsymbol{r}_{n})_{n\geq{}n_{0}+1}$ is a sequence of $(\mathcal{G}_{n})_{n\geq{}n_{0}+1}$-adapted random variables.

\begin{theoremx}[Almost sure convergence with ODE method, Theorem A.1 in \cite{Laruelle_Pages}]
	\label{thm_ODE_method}
	Assume that we have
	\begin{align*}
		\boldsymbol{r}_{n}\xrightarrow[n\to\infty]{\textrm{a.s.}} 0,\qquad
	\sup_{n\geq{}n_{0}}\mathbb{E}\left(||\Delta\boldsymbol{M}_{n+1}||_{2}^{2}\Big|\mathcal{G}_{n}\right)<\infty\textrm{ a.s.}
	\end{align*}
	and the sequence $(\gamma_{n})_{n\geq{}n_{0}}$ satisfies the following assumptions
	\begin{align*}
		\sum_{n=n_{0}}^{\infty}\gamma_{n}=\infty\qquad\textrm{and}\qquad\sum_{n=n_{0}}^{\infty}\gamma_{n}^{2}<\infty.
	\end{align*}
	We denote by $\Theta_{\infty}$ the almost sure limiting values of the sequence $(\vartheta_{n})_{n\geq{}n_{0}}$ as $n\to\infty$. Then $\Theta_{\infty}$ is almost surely a compact connected set.
	
	Let us have a look at the following ordinary differential equation:
	\begin{align}
		\label{eq_ode}
		\dot{\boldsymbol{\vartheta}}(t)&=-h(\boldsymbol{\vartheta}(t))\textrm{, where $t\geq{}t_{0}$}.
	\end{align}
	The flow of the above differential equation on $\Theta_{\infty}$ is $\varphi(t,t_{0},\boldsymbol{\vartheta}_{0})=\boldsymbol{\vartheta}(t)$, if $\boldsymbol{\vartheta}(t)$ is the solution of the this differential equation with initial value $\boldsymbol{\vartheta}(t_{0})=\boldsymbol{\vartheta}_{0}$.
	
	We assume that for every $\boldsymbol{\vartheta}_{0}\in\Theta_{\infty}$ the flow $\varphi(t,t_{0},\boldsymbol{\vartheta}_{0})$ is stable, i.e.\ for every $\varepsilon>0$ and $t_{1}>t_{0}$ there exists $\delta>0$, such that
	\begin{align*}
		\textrm{if } |\boldsymbol{\tau}-\varphi(t_{1},t_{0},\boldsymbol{\vartheta}_{0})|<\delta\textrm{, then } |\varphi(t,t_{1},\boldsymbol{\tau})-\varphi(t,t_{0},\boldsymbol{\vartheta}_{0})|<\varepsilon\textrm{ for every }t\geq{}t_{1}.
	\end{align*}
	If $\boldsymbol{\vartheta}^{*}\in\Theta_{\infty}$ is a uniformly stable equilibrium on $\Theta_{\infty}$ of the ordinary differential equation defined in \eqref{eq_ode}, i.e.\
	\begin{align*}
		\sup_{\boldsymbol{\vartheta}_{0}\in\Theta_{\infty}}|\varphi(t,t_{0},\boldsymbol{\vartheta}_{0})-\boldsymbol{\vartheta}^{*}|\xrightarrow[t\to\infty]{}0,
	\end{align*}
	then we have
	\begin{align*}
		\boldsymbol{\vartheta}_{n}\xrightarrow[n\to\infty]{a.s.}\boldsymbol{\vartheta}^{*}.
	\end{align*}
\end{theoremx}

\begin{proof}[Proof of Theorem $\ref{thm_asymptotic_urn_composition}$]
	We will prove the almost sure convergence of the normalized composition vector by using the ODE method. This proof is the generalization of the similar proof of Laruelle and Pag\`es in \cite{Laruelle_Pages}.
	
	For every $n\geq{}1$ we have
	\begin{align}
	\label{eq_urn_dynamics_extended}
		\boldsymbol{C}_{n+1}&=\boldsymbol{C}_{n}+\sum_{i=1}^{M}\boldsymbol{R}_{n+1}^{(i)}\boldsymbol{\chi}_{n+1}^{(i)}=\boldsymbol{C}_{n}+\mathbb{E}\left(\sum_{i=1}^{M}\boldsymbol{R}_{n+1}^{(i)}\boldsymbol{\chi}_{n+1}^{(i)}\Bigg|\mathcal{G}_{n}\right)+\Delta\boldsymbol{M}_{n+1},
	\end{align}
	where
	\begin{align*}
		\Delta\boldsymbol{M}_{n+1}&=\sum_{i=1}^{M}\boldsymbol{R}_{n+1}^{(i)}\boldsymbol{\chi}_{n+1}^{(i)}-\mathbb{E}\left(\sum_{i=1}^{M}\boldsymbol{R}_{n+1}^{(i)}\boldsymbol{\chi}_{n+1}^{(i)}\Bigg|\mathcal{G}_{n}\right).
	\end{align*}
	Recall that the generating matrices are defined as $\boldsymbol{H}_{n}^{(i)}=\mathbb{E}\left(\boldsymbol{R}_{n}^{(i)}\Big|\mathcal{G}_{n-1}\right)$. By using this and assumption \textbf{(U3)} on the conditional independence of $\boldsymbol{R}_{n+1}^{(i)}$ and $\boldsymbol{\chi}_{n+1}^{(i)}$, we have
	\begin{align*}
		\mathbb{E}\left(\sum_{i=1}^{M}\boldsymbol{R}_{n+1}^{(i)}\boldsymbol{\chi}_{n+1}^{(i)}\Bigg|\mathcal{G}_{n}\right)&=\sum_{i=1}^{M}\mathbb{E}\left(\boldsymbol{R}_{n+1}^{(i)}\boldsymbol{\chi}_{n+1}^{(i)}\Big|\mathcal{G}_{n}\right)\\
		&=\sum_{i=1}^{M}\sum_{j=1}^{N}\mathbb{E}\left(\boldsymbol{R}_{n+1}^{(i)}\mathrm{Ind}(\boldsymbol{\chi}_{n+1}^{(i)}=\boldsymbol{e}_{j})\boldsymbol{e}_{j}\Big|\mathcal{G}_{n}\right)\\
		&=\sum_{i=1}^{M}\left[\sum_{j=1}^{N}\mathbb{E}\left(\boldsymbol{R}_{n+1}^{(i)}\Big|\mathcal{G}_{n}\right)\mathbb{P}\left(\boldsymbol{\chi}_{n+1}^{(i)}=\boldsymbol{e}_{j}\Big|\mathcal{G}_{n}\right)\boldsymbol{e}_{j}\right]\\
		&=\sum_{i=1}^{M}\left[\boldsymbol{H}_{n+1}^{(i)}\sum_{j=1}^{N}\frac{C_{n,j}}{s(\boldsymbol{C}_{n})}\boldsymbol{e}_{j}\right]=\left(\sum_{i=1}^{M}\boldsymbol{H}_{n+1}^{(i)}\right)\frac{\boldsymbol{C}_{n}}{s(\boldsymbol{C}_{n})}.
	\end{align*}
	
	By normalizing equation \eqref{eq_urn_dynamics_extended}, we have
	\begin{align}
		\label{eq_normalized_urn_evolution}
		\frac{\boldsymbol{C}_{n+1}}{s(\boldsymbol{C}_{n+1})}&=\frac{\boldsymbol{C}_{n}}{s(\boldsymbol{C}_{n})}+\frac{1}{s(\boldsymbol{C}_{n+1})}\left[\left(\sum_{i=1}^{M}\boldsymbol{H}_{n+1}^{(i)}\right)-M\boldsymbol{I}_{N}\right]\frac{\boldsymbol{C}_{n}}{s(\boldsymbol{C}_{n})}+\frac{\Delta\boldsymbol{M}_{n+1}}{s(\boldsymbol{C}_{n+1})},
	\end{align}
	where $\boldsymbol{I}_{N}$ is the $N$-dimensional identity matrix. To verify the above reformulation we can check that
	\begin{align*}
		\frac{\boldsymbol{C}_{n}}{s(\boldsymbol{C}_{n})}-\frac{1}{s(\boldsymbol{C}_{n+1})}\cdot{}M\boldsymbol{I}_{N}\frac{\boldsymbol{C}_{n}}{s(\boldsymbol{C}_{n})}&=\frac{\boldsymbol{C}_{n}}{s(\boldsymbol{C}_{n+1})},
	\end{align*}
	by using the fact that $s(\boldsymbol{C}_{n+1})=s(\boldsymbol{C}_{n})+M$.
	
	
	Let us define $\widetilde{\boldsymbol{C}}_{n}=\boldsymbol{C}_{n}/s(\boldsymbol{C}_{n})$. Equation \eqref{eq_normalized_urn_evolution} can be rewritten as the canonical stochastic approximation process in the following way:
	\begin{align*}
		\widetilde{\boldsymbol{C}}_{n+1}&=\widetilde{\boldsymbol{C}}_{n}+\frac{1}{s(\boldsymbol{C}_{n+1})}\left[\left(\sum_{i=1}^{M}\boldsymbol{H}_{n+1}^{(i)}\right)-M\boldsymbol{I}_{N}\right]\widetilde{\boldsymbol{C}}_{n}+\frac{\Delta\boldsymbol{M}_{n+1}}{s(\boldsymbol{C}_{n+1})}\\
		&=\widetilde{\boldsymbol{C}}_{n}-\frac{1}{s(\boldsymbol{C}_{n+1})}M\left(\boldsymbol{I}_{N}-\boldsymbol{H}\right)\widetilde{C}_{n}+\frac{1}{s(\boldsymbol{C}_{n+1})}\left(\Delta\boldsymbol{M}_{n+1}+\boldsymbol{r}_{n+1}\right)
	\end{align*}
	with step size $\gamma_{n}=1/s(\boldsymbol{C}_{n})$ and the error term is defined as
	\begin{align*}
		\boldsymbol{r}_{n+1}&=\left[\left(\sum_{i=1}^{M}\boldsymbol{H}_{n+1}^{(i)}\right)-M\boldsymbol{H}\right]\widetilde{\boldsymbol{C}}_{n}.
	\end{align*}
	
	To apply the ODE method we need to check the assumptions of Theorem \ref{thm_ODE_method}.
	
	Since $\widetilde{\boldsymbol{C}}_{n}$ is bounded a.s., by using assumption \textbf{(U7)} we have $\boldsymbol{r}_{n}\xrightarrow[n\to\infty]{a.s.}0$. Notice that, we have
	\begin{align*}
		\Bigg|\Bigg|\sum_{i=1}^{M}\boldsymbol{R}_{n+1}^{(i)}\boldsymbol{\chi}_{n+1}^{(i)}\Bigg|\Bigg|_{2}^{2}\leq{}M^{2},
	\end{align*}
	since the number of balls added in a step equals to $M$, which is fixed. Consequently we have
	\begin{align*}
		\sup_{n\geq{}1}\mathbb{E}\left(\Bigg|\Bigg|\sum_{i=1}^{M}\boldsymbol{R}_{n+1}^{(i)}\boldsymbol{\chi}_{n+1}^{(i)}\Bigg|\Bigg|_{2}^{2}\Bigg|\mathcal{G}_{n}\right)<\infty,
	\end{align*}
	thus we obtain that $\sup_{n\geq{}1}\mathbb{E}\left(||\Delta\boldsymbol{M}_{n+1}||_{2}^{2}\Big|\mathcal{G}_{n}\right)<\infty$ almost surely.
	
	It is obvious that the (almost sure) limiting values of $\widetilde{\boldsymbol{C}}_{n}$ as $n\to\infty$ are in the $N$-dimensional simplex denoted by $\mathcal{S}=\left\{u\in\mathbb{R}_{+}^{N}|s(u)=1\right\}$. Let us have a look at the following ordinary differential equation:
	\begin{align*}
		\dot{y}&=-M(\boldsymbol{I}_{N}-\boldsymbol{H})y,
	\end{align*}
	where $y:\mathbb{R}^{N}\to\mathbb{R}^{N}$ is a differentiable function. By using assumption \textbf{(U7)} we obtain that $\boldsymbol{v}^{*}_{\boldsymbol{H}}$ is the unique zero of the function $y\mapsto-M(\boldsymbol{I}_{N}-\boldsymbol{H})Y$ on $y\in\mathcal{S}$. Let us take the restriction of the above differential equation to the set $\mathcal{V}_{0}=\left\{u\in\mathbb{R}_{+}^{N}|s(u)=0\right\}$. By using assumption \textbf{(U7)} we conclude that the left eigenvalues of $M(\boldsymbol{I}_{N}-\boldsymbol{H})$ have positive real part. As a consequence we get that $\boldsymbol{v}^{*}_{\boldsymbol{H}}$ is a uniformly stable equilibrium of the equation on $\mathcal{S}$. By using the ODE method we conclude that
	\begin{align*}
		\frac{\boldsymbol{C}_{n}}{s(\boldsymbol{C}_{n})}\xrightarrow[n\to\infty]{a.s.}\boldsymbol{v}^{*}_{\boldsymbol{H}}.
	\end{align*}
\end{proof}

\section{Proof of the main theorem}
First, we need to prove the following lemma on the asymptotic proportion of edges of different types. This is where we use Theorem \ref{thm_asymptotic_urn_composition} on the urn models.
\begin{lemma}
	\label{lemma_asymptotic_edge_composition}
	In the perturbed Barab\'asi--Albert random graph, if we assume that $\boldsymbol{F}\in(0,1)^{N\times{}N}$, then for every $l\in[N]$ we have $\psi_{n}^{(l)}=\frac{|E_{n}^{(l)}|}{|E_{n}|}\to\psi^{(l)}$ almost surely as $n\to\infty$, where $\psi^{(l)}\in(0,1)$ is a deterministic constant. If we also assume that $\boldsymbol{F}=(\varepsilon_{k,l})_{k,l=1}^{N}$ is symmetric, then for every $l\in[N]$ we have $\psi^{(l)}=\frac{1}{N}$.
\end{lemma}

\begin{proof}
	In the perturbed Barab\'asi--Albert model, we can use the following urn model to understand the asymptotic composition of the number edges of type $l$ for every $l\in[N]$. Let us have $\boldsymbol{C}_{0}=(|E_{0}^{(l)}|,l\in[N])$ and for every $n\geq{}1$ and $i\in[M]$ we define $\boldsymbol{R}_{n}^{(i)}=(\tau_{n;k,l}^{(i)})_{k,l=1}^{N}$, where $\tau_{n;k,l}^{(i)}$ is a Bernoulli distributed random variable with expectation equal to $\varepsilon_{l,k}^{(n)}$, furthermore we assume that for every $l\in[N]$ we have $\sum_{k=1}^{N}\tau_{n;k,l}^{(i)}=1$ and the columns of the matrix $\boldsymbol{R}_{n}^{(i)}$ are independent of each other. Clearly, we have
	\begin{align*}
		\boldsymbol{H}_{n}=\mathbb{E}\left(\boldsymbol{R}_{n}^{(1)}\big|\mathcal{G}_{n-1}\right)=\boldsymbol{F}_{n}^{T}.
	\end{align*}
	To apply Theorem \ref{thm_asymptotic_urn_composition} we have to check the assumptions of the general urn model and find $\boldsymbol{v}_{\boldsymbol{H}}^{*}$ to complete the proof of Lemma \ref{lemma_asymptotic_edge_composition}.
	
	Assumption \textbf{(U1)} holds due to the fact that there is at least one edge of each type in the initial configuration of the perturbed Barab\'asi--Albert random graph. By the dynamics of the model assumptions \textbf{(U2)}-\textbf{(U3)} hold (recall that we do not update the degrees of the vertices until the end of the steps). Assumptions \textbf{(U4)}-\textbf{(U6)} hold because of the choice of $\boldsymbol{R}_{n}^{(i)}$. Notice that in this case $\gamma_{1}=\gamma_{2}=1$. For assumption \textbf{(U7)} we need to show that there exists an irreducible $N\times{}N$ matrix denoted by $\boldsymbol{H}$ such that $\boldsymbol{H}_{n}\xrightarrow[n\to\infty]{a.s.}\boldsymbol{H}$. In the perturbed Barab\'asi--Albert random graph we assumed that $\boldsymbol{H}_{n}^{T}=\boldsymbol{F}_{n}\to\boldsymbol{F}$ in every entry. Since $\boldsymbol{F}$ is irreducible, we can choose $\boldsymbol{H}=\boldsymbol{F}^{T}$. Hence all assumptions of Theorem \ref{thm_asymptotic_urn_composition} hold.
	
	The normalized (right) eigenvector of $\boldsymbol{H}$ corresponding to the eigenvalue with the largest real part is $\boldsymbol{v}_{\boldsymbol{H}}^{*}=\left(\psi^{(1)},\dots,\psi^{(N)}\right)$. Notice that this is also the normalized left eigenvector of $\boldsymbol{F}$ corresponding to the same eigenvalue. By using Theorem \ref{thm_asymptotic_urn_composition} we get the first part of the lemma.
	
	If we also assume that $\boldsymbol{F}$ is symmetric, then $\boldsymbol{F}$ is a double-stochastic matrix. It follows that for every $l\in[N]$ we have $\psi^{(l)}=\frac{1}{N}$.
\end{proof}

Now, we can prove our main result on the asymptotic degree distribution of the perturbed Barab\'asi--Albert random graph.


We summarize a general graph model with its assumptions \textbf{(GM1)}-\textbf{(GM5)} from \cite{Backhausz_Rozner}. We will see that the perturbed Barab\'asi--Albert model is a special case.

In the general model, we have the following dynamics, in the $n^{\textrm{th}}$ step:
\begin{enumerate}
	\item{}A new vertex $v_{n}$ is born.
	\item{}Vertex $v_{n}$ is randomly connected to some of the old vertices with a few edges.
	\item{}Every new edge gets a type randomly.
\end{enumerate}
In addition, we have some assumptions on the general model:
\begin{description}
	\item[(GM1)]For every $n$, we assume that the conditional distribution of the number of new edges of type $l$ connected to $v\in{}V_{n-1}$, conditionally with respect of $\mathcal{F}_{n-1}$, depends only on $\textbf{deg}_{n-1}(v)$ for every $l\in[N]$. By using this assumption, we denote by $p_{\boldsymbol{d}}^{(n)}(\boldsymbol{i})$, where $\boldsymbol{i}=(i_{1},\dots,i_{N})$, the conditional probability that a vertex with generalized degree $\boldsymbol{d}$ gets exactly $i_{l}$ new edges of type $l$ for every $l\in[N]$, conditionally with respect to $\mathcal{F}_{n-1}$. In other words, for every existing vertex the probability of having a new edge only depends on the actual degree of the vertex.
	\item[(GM2)]For every $\boldsymbol{d}\in\mathbb{N}^{N}$, there exists $\delta>0$ and $C>0$, such that
	\begin{align*}
	\mathbb{E}\left(|X_{n}(\boldsymbol{d})-X_{n-1}(\boldsymbol{d})|^{2}\big|\mathcal{F}_{n-1}\right)&\leq{}Cn^{1-\delta}
	\end{align*}
	holds almost surely for every $n$. This means that the difference of the number of vertices with generalized degree $\boldsymbol{d}$ in the actual step and in the previous step is bounded in some sense.
	\item[(GM3)]For every $\boldsymbol{d}\in\mathbb{N}^{N}$, we define the sequence $\left(u_{n}(\boldsymbol{d})\right)_{n=1}^{\infty}$ by the following equality:
	\begin{align*}
	p_{\boldsymbol{d}}^{(n)}(\boldsymbol{0})&=1-\frac{u_{n}(\boldsymbol{d})}{n}.
	\end{align*}
	The sequence $\left(u_{n}(\boldsymbol{d})\right)_{n=1}^{\infty}$ is non-negative and predictable with respect to the filtration $\mathcal{F}$. We assume that there is a positive random variable denoted by $u(\boldsymbol{d})$, such that $u_{n}(\boldsymbol{d})\to{}u(\boldsymbol{d})$ almost surely as $n\to\infty$. This means that for a vertex with generalized degree $\boldsymbol{d}$ the probability of not receiving any new edge stabilizes in some sense.
	\item[(GM4)]For every $\boldsymbol{d}\in\mathbb{N}^{N}$ and for every $\boldsymbol{i}\in\mathbb{N}^{N}$, such that $s(\boldsymbol{i})\geq{}1$, we assume that there are families of non-negative random variables denoted by $r^{(l)}(\cdot)$ where $l\in[N]$, such that	
	\begin{align}
	\label{eq_GM4}
		\lim_{n\to\infty}np_{\boldsymbol{d}-\boldsymbol{i}}^{(n)}(\boldsymbol{i})&=
		\left\{
		\begin{array}{l l}
			r^{(l)}(\boldsymbol{d}-\boldsymbol{e}_{l}) & \textrm{if $\boldsymbol{i}=\boldsymbol{e}_{l}$} \\
			0 & \textrm{otherwise}
		\end{array}
	\right.
	\end{align}
	holds almost surely. Assumption \textbf{(GM4)} states that for an existing vertex the probability of receiving edges of different types to have generalized degree $\boldsymbol{d}$ has a non-trivial limit if and only if one of the edges of a given type is missing. That is, although it may happen that a vertex gets more than one edges in a step, this has probability of $O\left(\frac{1}{n^{2}}\right)$ and disappears from the asymptotic equations.
	\item[(GM5)]For every $n$ and for every $\boldsymbol{d}\in\mathbb{N}^{N}$, we denote by $q^{(n)}(\boldsymbol{d})$ the conditional probability that the new vertex $v_{n}$ is connected to the existing vertices with exactly $d_{l}$ edges of type $l$, conditionally with respect to $\mathcal{F}_{n-1}$. We assume that there exists a non-negative random variable denoted by $q(\boldsymbol{d})$, such that $q^{(n)}(\boldsymbol{d})\to{}q(\boldsymbol{d})$ almost surely as $n\to\infty$. This means that for the new vertex in the $n^{\textrm{th}}$ step the probability of having generalized degree $\boldsymbol{d}$ converges as $n\to\infty$.
\end{description}


To prove the existence of the asymptotic degree distribution, we can use Theorem 1 in \cite{Backhausz_Rozner}.

\begin{theoremx}[Theorem 1 in \cite{Backhausz_Rozner}]
	\label{thm_Backhausz_Rozner}
	If a sequence of random graphs with multi-type edges satisfies assumptions $\mathbf{(GM1)-(GM5)}$, then for every $\boldsymbol{d}\in\mathbb{N}^{N}$, we have
	\begin{align*}
		\lim_{n\to\infty}\frac{X_{n}(\boldsymbol{d})}{|V_{n}|}&=x(\boldsymbol{d})\textrm{ a.s.}
	\end{align*}
	Furthermore, for every $\boldsymbol{d}\in\mathbb{N}^{N}$, the random variables $x(\boldsymbol{d})$ satisfy the following recurrence equation:
	\begin{align*}
		x(\boldsymbol{d})&=\frac{1}{u(\boldsymbol{d})+1}\left[\sum_{l=1}^{N}r^{(l)}(\boldsymbol{d}-\boldsymbol{e}_{l})x(\boldsymbol{d}-\boldsymbol{e}_{l})+q(\boldsymbol{d})\right].
	\end{align*}
\end{theoremx}

For the analysis of the perturbed Barab\'asi--Albert random graph, we will use the above theorem, thus we have to show that it satisfies all the assumptions of the general model.

In the proof of the main theorem, we will use the following bound, which can be easily proved by Bonferroni's inequality.
\begin{lemma}
	\label{lemma_binomial_estimation}
	For every $n\geq{}1$ and $x\in[0,1]$, we have
	\begin{align*}
	|(1-x)^{n}-(1-nx)|&\leq{n \choose 2}x^{2}.
	\end{align*}
\end{lemma}

\textbf{Proof of Theorem \ref{thm_asymptotic_degree_dist}}.\ We need to check assumptions $\mathbf{(GM1)}$-$\mathbf{(GM5)}$ of the general model. We will use the following set of indices: for any $\boldsymbol{d}\in\mathbb{N}^{N}$ (where $s(\boldsymbol{d})\geq{}1$), we define
\begin{align*}
	\alpha(\boldsymbol{d})=\left\{\boldsymbol{i}=(i_{1},\dots,i_{N})^{T}\in\mathbb{N}^{N}:i_{l}\leq{}d_{l}\quad\forall{}l\in[N]\quad\textrm{and}\quad{}1\leq{}s(\boldsymbol{i})\leq{}M\right\}.
\end{align*}

For every $\boldsymbol{i}\in\mathbb{N}^{N}$ we define
\begin{align*}
	\beta(\boldsymbol{i})=\left\{\boldsymbol{I}=(i_{k,l})_{k,l=1}^{N}\in\mathbb{N}^{N\times{}N}:s\left(\boldsymbol{I}_{\bullet,l}\right)=i_{l}\quad\forall{}l\in[N]\right\}.
\end{align*}

By the dynamics of the perturbed Barab\'asi--Albert random graph, assumption $\mathbf{(GM1)}$ trivially holds.

To see that assumption $\mathbf{(GM2)}$ is satisfied, notice that for every $n$, we have
\begin{align*}
	\mathbb{E}\left[(X_{n}(\boldsymbol{d})-X_{n-1}(\boldsymbol{d}))^{2}\big|\mathcal{F}_{n-1}\right]&\leq{}M^{2}<\infty,
\end{align*}
because there are $M$ new edges.

For assumption $\mathbf{(GM3)}$, we need to show that $u_{n}(\boldsymbol{d})\to{}u(\boldsymbol{d})>0$ almost surely as $n\to\infty$, where
\begin{align*}
	1-\frac{u_{n}(\boldsymbol{d})}{n}&=p_{\boldsymbol{d}}^{(n)}(\boldsymbol{0})=\left(1-\frac{s(\boldsymbol{d})}{2|E_{n-1}|}\right)^{M}.
\end{align*}
To find the almost sure limit of $u_{n}(\boldsymbol{d})$ as $n\to\infty$, we can use the following formula:
\begin{align*}
	\left(1-\frac{s(\boldsymbol{d})}{2|E_{n-1}|}\right)^{M}&=1-\frac{Ms(\boldsymbol{d})}{2|E_{n-1}|}+\eta_{n}(\boldsymbol{d}),
\end{align*}
where
\begin{align*}
	\eta_{n}(\boldsymbol{d})&=\left(1-\frac{s(\boldsymbol{d})}{2|E_{n-1}|}\right)^{M}-\left[1-\frac{Ms(\boldsymbol{d})}{2|E_{n-1}|}\right].
\end{align*}
By using Lemma \ref{lemma_binomial_estimation} and the fact that $|E_{n}|\sim{}Mn$, we obtain that
\begin{align*}
	|\eta_{n}(\boldsymbol{d})|&\leq{M \choose 2}\left(\frac{s(\boldsymbol{d})}{2|E_{n-1}|}\right)^{2}\leq{}M^{2}\left(\frac{s(\boldsymbol{d})}{2|E_{n-1}|}\right)^{2}=o\left(\frac{1}{n}\right),
\end{align*}
which yields
\begin{align*}
	u_{n}(\boldsymbol{d})&=n[1-p_{\boldsymbol{d}}^{(n)}(\boldsymbol{0})]=n\left[1-\left(1-\frac{s(\boldsymbol{d})}{2|E_{n-1}|}\right)^{M}\right]\\
	&=n\left[1-\left(1-\frac{Ms(\boldsymbol{d})}{2|E_{n-1}|}+\eta_{n}(\boldsymbol{d})\right)\right]=n\left[\frac{Ms(\boldsymbol{d})}{2|E_{n-1}|}-\eta_{n}(\boldsymbol{d})\right]\\
	&=n\frac{Ms(\boldsymbol{d})}{2|E_{n-1}|}-n\eta_{n}(\boldsymbol{d})\to{}u(\boldsymbol{d})=\frac{s(\boldsymbol{d})}{2}>0
\end{align*}
almost surely as $n\to\infty$.

For assumption $\mathbf{(GM4)}$, we will show that for every $\boldsymbol{i}\in\mathbb{N}^{N}$, such that $s(\boldsymbol{i})\geq{}1$
\begin{align*}
	\lim_{n\to\infty}np_{\boldsymbol{d}-\boldsymbol{i}}^{(n)}(\boldsymbol{i})=\left\{
	\begin{array}{l l}
		\frac{1}{2}(\boldsymbol{d}-\boldsymbol{e}_{l})^{T}\boldsymbol{F}_{\bullet,l} & \textrm{if $\boldsymbol{i}=\boldsymbol{e}_{l}$,}\\
		0 & \textrm{otherwise}
	\end{array}
	\right.
\end{align*}
holds almost surely. Notice that if $\boldsymbol{d}_{k}-\boldsymbol{i}_{k}<0$ for any $k\in[N]$ or $\boldsymbol{d}-\boldsymbol{i}=\boldsymbol{0}$, then $p_{\boldsymbol{d}-\boldsymbol{i}}^{(n)}(\boldsymbol{i})=0$ for every $n$.

Let us fix $\boldsymbol{I}=(i_{k,l})_{k,l=1}^{N}\in\beta(\boldsymbol{i})$ where $\boldsymbol{i}\in\alpha(\boldsymbol{d})$. In the perturbed Barab\'asi--Albert random graph for a fixed vertex and for every $k,l\in[N]$ we denote by $i_{k,l}$ the number of edges connected to the given vertex which were originally of type $k$ and changed their types to $l$. In this case the value of $p_{\boldsymbol{d}-\boldsymbol{i}}^{(n)}(\boldsymbol{i})$ is given by
\begin{align*}
	p_{\boldsymbol{d}-\boldsymbol{i}}^{(n)}(\boldsymbol{i})&=\sum_{\boldsymbol{I}\in\beta(\boldsymbol{i})}\hat{p}_{\boldsymbol{d}-\boldsymbol{i}}^{(n)}(\boldsymbol{I}),
\end{align*}
where
\begin{align*}
	\hat{p}_{\boldsymbol{d}-\boldsymbol{i}}^{(n)}(\boldsymbol{I})&=\frac{M!}{\left(\prod_{k,l=1}^{N}i_{k,l}!\right)(M-s(\boldsymbol{i}))!}\left[\prod_{k,l=1}^{N}\left[\left(\frac{\boldsymbol{d}_{k}-\boldsymbol{i}_{k}}{2|E_{n-1}|}\right)\varepsilon_{k,l}^{(n)}\right]^{i_{k,l}}\right]\left[1-\frac{s(\boldsymbol{d}-\boldsymbol{i})}{2|E_{n-1}|}\right]^{M-s(\boldsymbol{i})}.
\end{align*}

First, let us fix $l\in[N]$. Similarly to the previous calculations, we can use the following formula 
\begin{align*}
	p_{\boldsymbol{d}-\boldsymbol{e}_{l}}^{(n)}(\boldsymbol{e}_{l})&=M\left[\sum_{k=1}^{N}\frac{\boldsymbol{d}_{k}-(\boldsymbol{e}_{l})_{k}}{2|E_{n-1}|}\varepsilon_{k,l}^{(n)}\right]\left(1-\frac{s(\boldsymbol{d})-1}{2|E_{n-1}|}\right)^{M-1}\\
	&=M\left[\sum_{k=1}^{N}\frac{\boldsymbol{d}_{k}-(\boldsymbol{e}_{l})_{k}}{2|E_{n-1}|}\varepsilon_{k,l}^{(n)}\right]\left(1-\frac{(M-1)(s(\boldsymbol{d})-1)}{2|E_{n-1}|}+\eta_{n}'(\boldsymbol{d})\right)
\end{align*}
where
\begin{align*}
	\eta_{n}'(\boldsymbol{d})&=\left(1-\frac{s(\boldsymbol{d})-1}{2|E_{n-1}|}\right)^{M-1}-\left[1-\frac{(M-1)(s(\boldsymbol{d})-1)}{2|E_{n-1}|}\right]
\end{align*}
and $(\boldsymbol{e}_{l})_{k}$ denotes the $k^{\textrm{th}}$ element of $\boldsymbol{e}_{l}$, i.e.\ $(\boldsymbol{e}_{l})_{k}=\left\{\begin{array}{l l}
	1 & \textrm{if $k=l$} \\
	0 & \textrm{if $k\neq{}l$}
\end{array}\right.$.

Again, by using Lemma \ref{lemma_binomial_estimation} and the fact that $|E_{n}|\sim{}Mn$, we get that
\begin{align*}
	|\eta_{n}'(\boldsymbol{d})|&\leq{M-1 \choose 2}\left(\frac{s(\boldsymbol{d})-1}{2|E_{n-1}|}\right)^{2}\leq{}(M-1)^{2}\left(\frac{s(\boldsymbol{d})-1}{2|E_{n-1}|}\right)^{2}=o\left(\frac{1}{n}\right).
\end{align*}
Recall that $\varepsilon_{k,l}^{(n)}\to\varepsilon_{k,l}\in(0,1)$ as $n\to\infty$ for every $k,l\in[N]$. We conclude that
\begin{align*}
	np_{\boldsymbol{d}-\boldsymbol{e}_{l}}^{(n)}(\boldsymbol{e}_{l})&=nM\left[\sum_{k=1}^{N}\frac{\boldsymbol{d}_{k}-(\boldsymbol{e}_{l})_{k}}{2|E_{n-1}|}\varepsilon_{k,l}^{(n)}\right]\left(1-\frac{(M-1)(s(\boldsymbol{d})-1)}{2|E_{n-1}|}+\eta_{n}'(\boldsymbol{d})\right)\\
	&\to\sum_{k=1}^{N}\frac{\boldsymbol{d}_{k}-(\boldsymbol{e}_{l})_{k}}{2}\varepsilon_{k,l}=\frac{(\boldsymbol{d}-\boldsymbol{e}_{l})^{T}\boldsymbol{F}_{\bullet,l}}{2}
\end{align*}
as $n\to\infty$.

Fix $\boldsymbol{i}\in\alpha(\boldsymbol{d})\setminus\left\{\boldsymbol{e}_{l},l\in[N]\right\}$. We need to prove that in this case $\lim_{n\to\infty}np_{\boldsymbol{d}-\boldsymbol{i}}^{(n)}(\boldsymbol{i})=0$. Recall that
\begin{align*}
	p_{\boldsymbol{d}-\boldsymbol{i}}^{(n)}(\boldsymbol{i})&=\sum_{I\in\beta(\boldsymbol{i})}\hat{p}_{\boldsymbol{d}}^{(n)}(\boldsymbol{I}).
\end{align*}
Because of the choice of $\boldsymbol{i}$, we have $s(\boldsymbol{i})=s(\boldsymbol{I})\geq{}2$. By using this and the fact that $|E_{n}|\sim{}Mn$, we conclude that
\begin{align*}
	\hat{p}_{\boldsymbol{d}-\boldsymbol{i}}^{(n)}(\boldsymbol{I})&=\frac{M!}{\left(\prod_{k,l=1}^{N}i_{k,l}!\right)(M-s(\boldsymbol{i}))!}\left[\prod_{k,l=1}^{N}\left[\left(\frac{\boldsymbol{d}_{k}-\boldsymbol{i}_{k}}{2|E_{n-1}|}\right)\varepsilon_{k,l}^{(n)}\right]^{i_{k,l}}\right]\left[1-\frac{s(\boldsymbol{d}-\boldsymbol{i})}{2|E_{n-1}|}\right]^{M-s(\boldsymbol{i})}\\
	&\leq\frac{M!}{\left(\prod_{k,l=1}^{N}i_{k,l}!\right)(M-s(\boldsymbol{i}))!}\left[\prod_{k,l=1}^{N}\left[\left(\frac{\boldsymbol{d}_{k}-\boldsymbol{i}_{k}}{2|E_{n-1}|}\right)\varepsilon_{k,l}^{(n)}\right]^{i_{k,l}}\right]=o\left(\frac{1}{n}\right).
\end{align*}
This shows that (\textbf{GM4}) holds with $r^{(l)}(\boldsymbol{d}-\boldsymbol{e}_{l})=\frac{1}{2}(\boldsymbol{d}-\boldsymbol{e}_{l})^{T}\boldsymbol{F}_{\bullet,l}$.

Finally, for assumption $\mathbf{(GM5)}$, we have to find the almost sure limit of $q^{(n)}(\boldsymbol{d})$ as $n\to\infty$. Recall that, in the $n^{\textrm{th}}$ step, every new edge will be of type $l$ with probability $\psi_{n}^{(l)}=\frac{|E_{n}^{(l)}|}{|E_{n}|}$.

In the perturbed Barab\'asi--Albert random graph we have
\begin{align*}
	q(\boldsymbol{d})&=\lim_{n\to\infty}q^{(n)}(\boldsymbol{d})=\lim_{n\to\infty}\left[\textrm{Ind}(M=s(\boldsymbol{d}))\sum_{\boldsymbol{D}\in\beta(\boldsymbol{d})}\hat{q}^{(n)}(\boldsymbol{D})\right],
\end{align*}
where
\begin{align*}
	\hat{q}^{(n)}(\boldsymbol{D})&=\frac{s(\boldsymbol{D})!}{\prod_{k,l=1}^{N}d_{k,l}!}\left[\prod_{k,l=1}^{N}\left(\psi_{n}^{(k)}\cdot\varepsilon_{k,l}^{(n)}\right)^{d_{k,l}}\right].
\end{align*}

Notice that $s(\boldsymbol{D})=s(\boldsymbol{d})$. By using Lemma \ref{lemma_asymptotic_edge_composition}, the multinomial theorem and the fact that $\varepsilon_{k,l}^{(n)}\to\varepsilon_{k,l}$ almost surely as $n\to\infty$, we get that
\begin{align*}
	q(\boldsymbol{d})&=\textrm{Ind}(M=s(\boldsymbol{d}))\sum_{\boldsymbol{D}\in\beta(\boldsymbol{d})}\hat{q}(\boldsymbol{D})\\
	&=\textrm{Ind}(M=s(\boldsymbol{d}))\sum_{\boldsymbol{D}\in\beta(\boldsymbol{d})}\frac{s(\boldsymbol{D})!}{\prod_{k,l=1}^{N}d_{k,l}!}\left[\prod_{k,l=1}^{N}\left(\psi^{(k)}\cdot\varepsilon_{k,l}\right)^{d_{k,l}}\right]\\
	&=\textrm{Ind}(M=s(\boldsymbol{d}))M!\left[\prod_{l=1}^{N}\frac{1}{d_{l}!}\left(\sum_{k=1}^{N}\psi^{(k)}\cdot\varepsilon_{k,l}\right)^{d_{l}}\right].
\end{align*}

We conclude that
\begin{align*}
	u_{n}(\boldsymbol{d})&\to{}u(\boldsymbol{d})=\frac{s(\boldsymbol{d})}{2}\\
	q^{(n)}(\boldsymbol{d})&\to{}q(\boldsymbol{d})=\textrm{Ind}(M=s(\boldsymbol{d}))M!\left[\prod_{l=1}^{N}\frac{1}{d_{l}!}\left(\sum_{k=1}^{N}\psi^{(k)}\cdot\varepsilon_{k,l}\right)^{d_{l}}\right]
\end{align*}
as $n\to\infty$. For the quantity defined in equation \eqref{eq_GM4}, we have
\begin{align*}
	\lim_{n\to\infty}np_{\boldsymbol{d}-\boldsymbol{i}}^{(n)}(\boldsymbol{i})=\left\{
	\begin{array}{l l}
		\frac{(\boldsymbol{d}-\boldsymbol{e}_{l})^{T}\boldsymbol{F}_{\bullet,l}}{2} & \textrm{if $\boldsymbol{i}=\boldsymbol{e}_{l}$,}\\
		0 & \textrm{otherwise}
	\end{array}
	\right.
\end{align*}
almost surely, that is, $r^{(l)}(\boldsymbol{d}-\boldsymbol{e}_{l})=\frac{1}{2}(\boldsymbol{d}-\boldsymbol{e}_{l})^{T}\boldsymbol{F}_{\bullet,l}$.

Applying Theorem \ref{thm_Backhausz_Rozner}, we get Theorem \ref{thm_asymptotic_degree_dist}.\hfill$\square$

\section*{Acknowledgements}
This research was partially supported by Pallas Athene Domus Educationis Foundation. The views expressed are those of the authors’ and do not necessarily reflect the official opinion of Pallas Athene Domus Educationis Foundation. The first author was supported by the Bolyai Research Grant of the Hungarian Academy of Sciences.

The authors thank the referee for the review and appreciate the comments and suggestions which contributed to the improvements of the article.

\end{document}